\documentclass[11pt]{article}
\usepackage{amsmath}
\usepackage{amssymb}
\usepackage{amsfonts}
\usepackage{mathrsfs}
\usepackage{cite}
\usepackage{cleveref}
\newtheorem{theorem}{Theorem}[section]

\newtheorem{lemma}{Lemma}[section]

\newtheorem{proposition}{Proposition}[section]
\newtheorem{remark}{Remark}[section]

\numberwithin{equation}{section}

\newenvironment{proof}{{\noindent \bf Proof.}}{\hfill $\Box$}
\newenvironment{proof3.1}{{\noindent \bf Proof of Theorem 3.1.}}{\hfill $\Box$}
\allowdisplaybreaks[4]

\begin{document}

\setlength{\baselineskip}{16pt}{\setlength\arraycolsep{2pt}}

\title{Polynomial stability of piezoelectric beams with magnetic effect and tip body }

\author{Yanning An,\ \ Wenjun Liu\footnote{Corresponding author.  \ \   Email address: wjliu@nuist.edu.cn (W. J. Liu).} \ \ and   \ Aowen Kong  \medskip\\School of Mathematics and Statistics,
Nanjing University of Information Science \\
and Technology, Nanjing 210044, China}


\date{}
\maketitle

\begin{abstract}
In this paper, we consider a dissipative system of one-dimensional piezoelectric beam with magnetic effect and a tip load at the free end of the beam, which is modeled as a special form of double boundary dissipation. Our main aim is to study the well-posedness and asymptotic behavior of this system. By introducing two functions defined on the right boundary, we first transform the original problem into a new abstract form, so as to  show the well-posedness of the system by using Lumer-Philips theorem. We then divide the original system into a conservative system and an auxiliary system, and show that the auxiliary problem generates a compact operator. With the help of Wely's theorem, we obtain that the system is not exponentially stable. Moreover, we prove the polynomial stability of the system by using a result of Borichev
and Tomilov (Math. Ann. {\bf 347} (2010), 455--478).
\end{abstract}


\maketitle

\section{Introduction }

Piezoelectric materials are materials that can exchange mechanical energy, electrical energy and nuclear energy in motion. Their structures are generally composed of beams or slabs. Due to the advantages like small size, high power density, fast response time, large mechanical force and high resolution, they have more and more application prospects in many fields, such as the latest cutting-edge applications: cardiac pacemaker \cite{dag2014}, course changing bullet, structural health monitoring \cite{giu2008}, nano locator \cite{guzhu2016}, ultrasonic imaging device, ultrasonic welding and cleaning device, energy collection \cite{ertinm2008}.
The piezoelectric effect usually is shown as two types. One is to generate charge in the interior by applying mechanical force, which is called direct piezoelectric effect \cite{ka1969,publi1996}. Another is from the external electric field through its internal mechanical stress, which is called reverse piezoelectric effect. Due to the asymmetry of crystals, the above two effects have the same origin \cite{galdin2000}. In the piezoelectric beam, which constitutes the electronic device, the mechanical disturbance responds in the form of electricity. When piezoelectric materials are integrated into components of electronic circuits, the mechanical effects on structures are also very important when they are interfered by electrical, magnetic or electromagnetic properties. There are three main ways to drive piezoelectric materials in such electronic devices: to supply voltage, current or charge to the electrodes. Therefore, it is very important to describe the interaction of these three effects (mechanical, electrical and magnetic) for understanding the stability conditions of these systems (\cite{bsw1996,dwg2010,yang2006}). The equation of piezoelectric beam with magnetic effect is based on the description of electromagnetic coupling by Maxwell equation and the mechanical behavior of beam by Mindlin-Timoshenko theory ({\cite{bsw1996,dwg2010}}).

Let us refer to several previous works on the stability results for the piezoelectric models. In \cite{mo2013,mo2014}, Morris and \"{O}zer considered the effects of three effects (mechanical, electrical and magnetic) for the first time. They studied the dissipative systems
\begin{equation} \label{other1.2}
\begin{aligned}
&\rho v_{tt}-\alpha v_{xx}+\gamma\beta p_{xx}=0,&(x,t)\in(0,L)\times (0,T),\\
&\mu p_{tt}-\beta p_{xx}+\gamma\beta v_{xx}=0,&(x,t)\in(0,L)\times (0,T),
\end{aligned}
\end{equation}
with boundary conditions
\begin{equation} \label{other1.2b}
\begin{aligned}
&v(0,t)=\alpha v_{x}(L,t)-\gamma\beta p_{x}(L,t)=0,& t\in (0,T),\\
&p(0,t)=\beta p_{x}(L,t)-\gamma\beta v_{x}(L,t)+V(t)=0, & t\in (0,T),
\end{aligned}
\end{equation}
where $v(x,t)$ and $p(x,t)$ represent respectively the displacement of the upper and lower plates, and $\rho,\mu,\alpha,\beta$  $\gamma$ denote respectively the mass density per unit volume, the magnetic permeability, the elastic stiffness, the beam coefficient of impermeability and the piezoelectric coefficient. The relationship between $\alpha,\beta$ and $\gamma$ is given as $\alpha=\alpha_{1}+\gamma^{2}\beta$, where $\alpha_{1}>0$ represents the elastic stiffness of the model derived from the electrostatic and quasi-static methods of Euler Bernoulli small displacement (see example \cite{mo2014}). And $V(t)=\frac{p_{t}(L,t)}{h}$ is the prescribed voltage on the beam electrodes. The authors showed that system \eqref{other1.2}-\eqref{other1.2b} with only one boundary control was not exponentially stable.

In \cite{rfajm2019}, Ramos et al. studied a one-dimensional system of piezoelectric beams with magnetic effect, the system is shown as
\begin{equation} \label{other1.3}
\begin{aligned}
&\rho v_{tt}-\alpha v_{xx}+\gamma\beta p_{xx}=0,&(x,t)\in(0,L)\times (0,T),\\
&\mu p_{tt}-\beta p_{xx}+\gamma\beta v_{xx}=0,&(x,t)\in(0,L)\times (0,T),
\end{aligned}
\end{equation}
with boundary conditions
\begin{equation} \label{other1.3b}
\begin{aligned}
&v(0,t)=\alpha v_{x}(L,t)-\gamma\beta p_{x}(L,t)+\frac{\xi_{1}}{h}v_{t}(L,t)=0,& t\in (0,T),\\
&p(0,t)=\beta p_{x}(L,t)-\gamma\beta v_{x}(L,t)+\frac{\xi_{2}}{h}p_{t}(L,t)=0,& t\in (0,T),
\end{aligned}
\end{equation}
where $\xi_{i},i=1,2$ are positive constant feedback gains. By using multiplier method, the authors proved that the system is exponentially stable, and obtained that the exponential stability is equivalent to the exact observability at the boundary.

Yang and Wang \cite{yw2017} studied the piezoelectric layer actuated by a voltage source without magnetic effects. They modeled the system
\begin{equation} \label{other1.1}
\left\{
\begin{aligned}
&u_{tt}-\alpha_{u}u_{xx}-G_{u}z=0,&(x,t)\in(0,L)\times (0,T),\\
&v_{tt}-\alpha_{v}v_{xx}+G_{v}z=0,&(x,t)\in(0,L)\times (0,T),\\
&y_{tt}-Iy_{xxtt}+y_{xxxx}-G_{1}z_{x}=0,&(x,t)\in(0,L)\times (0,T),\\
&\phi=\frac{1}{h}(-\varphi+\psi+H\omega_{x}),&(x,t)\in(0,L)\times (0,T),
\end{aligned}
\right.
\end{equation}
with the boundary conditions
\begin{equation} \label{other1.1b}
\left\{
\begin{aligned}
&u(0,t)=v(0,t)=y(0,t)=y_{x}(0,t)=0, & t\in (0,T),\\
&u_{x}(L,t)=l_{1}u_{t}(L,t),& t\in (0,T),\\
&v_{x}(L,t)=-l_{2}v_{t}(L,t),& t\in (0,T),\\
&y_{xx}(L,t)=-l_{3}y_{xt}(L,t),& t\in (0,T),\\
&Iy_{xtt}(L,t)-y_{xxx}(L,t)+G_{1}z(1,t)=0,& t\in (0,T),
\end{aligned}
\right.
\end{equation}
where $u(x,t)$ and $v(x,t)$ are the longitudinal displacements of the bottom layer and the top layer, respectively, and $y(x,t)$ is the transverse displacement of each floor (the transverse displacement of the three floors is regarded as equal). By using the method of Riesz basis, the authors obtained the exponential stability of system \eqref{other1.1}-\eqref{other1.1b}.

Recently, some researchers have studied Timoshenko system with tip body and hybrid system with tip load damped, see \cite{rvma2017,var2020,mun2015}.
In industry, many piezoelectric beam devices are in the form of a boundary with a tip body, such as the electrostatic energy harvester mentioned in reference \cite{erturk2011,sun2018}. The tip body has mass, so its appearance will bring tip inertia, which will affect the stability of the system. Therefore, it is necessary to study the piezoelectric beam system with tip load.

In this paper, we shall study the polynomial stability of a piezoelectric beam system with magnetic effect and tip body. Considering a piezoelectric beam with a tip load, the beam is clamped at $x = 0$, and the tip is fixed at $x = L$. The center of mass of the tip is the connection point between the tip body and the piezoelectric beam plate. We assume that the beam interacts with the tip body, and the force of the vibrating beam moves to the end load according to Newton's law. By using the feedback boundary force control to the displacement velocity at $x = L$, dissipation is introduced into the piezoelectric system. Then the coupling model is given by
\begin{equation} \label{problem1.1}
\begin{aligned}
&\rho V_{tt}-\alpha V_{xx}+\gamma\beta P_{xx}=0, &(x,t)\in(0,L)\times (0,T),\\
&\mu P_{tt}-\beta P_{xx}+\gamma\beta V_{xx}=0,  &(x,t)\in(0,L)\times (0,T),
\end{aligned}
\end{equation}
with the double boundary conditions
\begin{equation} \label{1.2}
\begin{aligned}
&V(0,t)=P(0,t)=0,& t\in (0,T),\\
&\alpha V_{x}(L,t)-\gamma\beta P_{x}(L,t)+\xi_{1}V_{t}(L,t)+m_{1}V_{tt}(L,t)=0,& t\in (0,T),\\
&\beta P_{x}(L,t)-\gamma\beta V_{x}(L,t)+\xi_{2}P_{t}(L,t)+m_{2}P_{tt}(L,t)=0,& t\in (0,T),
\end{aligned}
\end{equation}
and the initial conditions
\begin{align}\label{1.3}
\left(V(x,0),V_{t}(x,0),P(x,0),P_{t}(x,0)\right)=\left(V_{0}(x),V_{1}(x),P_{0}(x),P_{1}(x)\right),x\in(0,L).
\end{align}
where $V(x,t), P(x,t)$ represent respectively the longitudinal displacements of the upper and lower plates, $\xi_{1},\xi_{2}$ are positive constant feedback gains, and $ m_{1},m_{2}$ are mass of tip load. Equations \eqref{1.2}${_{2}}$ and \eqref{1.2}${_{3}}$ are obtained by the force balance at the end $x = L$. The third term in the boundary conditions at the end $x=L$ represents the influence of the magnetic effect, and the first two terms represent shear force.

In this paper, we study the stability of a piezoelectric beam with a tip on both plates.  After getting the well-posedness of the system by using the classical Lumer-Philips theorem, we will start to analyze the stability of the system. By dealing with the resolvent equation of system \eqref{problem1.1}-\eqref{1.3}, we obtain an observable inequality. Then, combined with Borichev and Tomilov theorem \cite{bor2010}, we will prove that the system is polynomial stable.
The difficulty of stability analysis lies in how to obtain that the system is lack of uniform stability. Because the system has only two equations, and the partial derivatives in the $x$ direction of the equations are second order. So it is difficult to construct a suitable function sequence and use the usual Gearhart-Herbst-Pr\"{u}ss-Huang theorem as in \cite{lkl2020,liuzhao2019} to prove that the system is not exponentially stable. To overcome this difficulty, we divide the original system into a conservative system and an auxiliary system, and show that the auxiliary problem generates a compact operator. By using the Wely's theorem \cite{weyl1910}, we get that the growth bound of the original system is $0$, that is, the system is not uniformly exponential stable. Since uniform stability and uniform exponential stability are equivalent in strongly continuous semigroups, we show that the system is not uniformly stable. Some typical problems  can be found in references \cite{liuzhao2019,bm2020,HWW2020,HS2021,jp2019,kam2016,lx2017,mmrpv2019,maycmg2019,p2018,tita2018,lkl2020,liuzhao2019,alv2020,mus2021,car2019,mess2018}.

The structure of this paper is as follows. In the next section, we will give the well-posedness of system \eqref{problem1.1}-\eqref{1.3}. In Section 3, we will show the lack of uniformly stability. Finally, we will get the polynomial stability of the system in Section 4.

\section{Well-posedness}

In this section, we give a well-posedness result for problem \eqref{problem1.1}-\eqref{1.3} by using a semigroup approach.

To define the semigroup associated with \eqref{problem1.1}-\eqref{1.3}, we introduce two new functions which are defined by
\begin{align}\label{2.1}
u(t)=V_{t}(L,t) \quad\quad \mathrm{and} \quad\quad \eta(t)=P_{t}(L,t), \quad\quad t>0,
\end{align}
respectively, with
\begin{align}\label{2.2}
u(0)=V_{1}(L)=u_{0} \quad\quad\mathrm{and}\quad\quad \eta (0)=P_{1}(L)=\eta_{0}.
\end{align}
By using the definition of $u,\eta $, we can change system \eqref{problem1.1}-\eqref{1.3}  to
\begin{align}
&\rho V_{tt}-\alpha V_{xx}+\gamma\beta P_{xx}=0,\quad (x,t)\in  (0,L)\times (0,T), \label{2.3}\\
&\mu P_{tt}-\beta P_{xx}+\gamma\beta V_{xx}=0,   \quad (x,t)\in(0,L)\times (0,T),\label{2.4}
\end{align}
with the boundary conditions
\begin{align}
& V(0,t)=\alpha V_{x}(L,t)-\gamma\beta P_{x}(L,t)+\xi_{1} u(t)+m_{1}u_{t}(t)=0,&t\in(0,+\infty), \label{2.5}\\
& P(0,t)=\beta P_{x}(L,t)-\gamma\beta V_{x}(L,t)+\xi_{2} \eta(t)+m_{2}\eta_{t}(t)=0,& t\in(0,+\infty), \label{2.6}
\end{align}
and the initial conditions
\begin{align}\label{2.7}
\left(V(x,0), V_{t}(x,0), P(x,0), P_{t}(x,0),u(0),\eta (0)\right)=\left(V_{0},V_{1},P_{0},P_{1},u_{0},\eta_{0}\right)\quad x \in (0,L).
\end{align}

The energy of system \eqref{2.3}-\eqref{2.7} is given by
\begin{align}\label{2.8}
E(t)=\frac{1}{2}\int^{L}_{0}\left[\rho\left|V_{t}\right|^{2}+\alpha_{1}\left|V_{x}\right|^{2}+\mu\left|P_{t}\right|^{2}+\beta\left|\gamma V_{x}-P_{x}\right|^{2}\right] dx+\frac{m_{1}}{2}|u|^{2}+\frac{m_2}{2}|\eta|^{2}.
\end{align}
Multiplying \eqref{2.3}, \eqref{2.4} by $ V_{t}$ and $P_{t}$ respectively, and using the boundary conditions \eqref{2.5}-\eqref{2.6}, we get
\begin{align}\label{2.9}
\frac{d}{dt}E(t)=-\xi_{1}\left|V_{t}(L,t)\right|^{2}-\xi_{2}\left|P_{t}(L,t)\right|^{2}.
\end{align}

Let us define the space $\mathcal{H}$ as
\begin{align*}
\mathcal{H}:=H^{1}_{*}(0,L)\times {L}^{2}(0,L)\times {{H}}^{1}_{*}(0,L)\times {{L}}^{2}(0,L)\times \mathbb{C}\times \mathbb{C},
\end{align*}
for $ {{H}}^{1}_{*}(0,L)=\left\{f\in{{H}}^{1}(0,L):f(0)=0\right\}$, equipped with the inner product
\begin{align*}
\langle{U}_{1},{U}_{2}\rangle_{\mathcal{H}}=&\int^{L}_{0}\left[\rho\Phi_{1}\overline{\Phi}_{2}+\mu\Theta_{1}\overline{\Theta}_{2}+\alpha_{1}V_{1,x}\overline{V}_{2,x}+\beta(\gamma V_{1,x}-P_{1,x})\overline{(\gamma V_{2,x}-P_{2,x})}\right]dx \\
&+m_{1}u_{1}\overline{u}_{2}+m_{2}\eta_{1}\overline{\eta}_{2},
\end{align*}
where ${{U}}_{i}=\left(V_{i},\Phi_{i}, P_{i}, \Theta_{i}, u_{i}, \eta_{i}\right)\in {H},i=1,2$.
Set the vector function ${{U}}=\left(V,V_{t}, P, P_{t}, u, \eta\right)^{T}$, then system \eqref{2.3}-\eqref{2.7} can be written as
\begin{equation} \label{2.10}
\left\{
\begin{aligned}
&{{U}}_{t}={\mathcal{A}}{{U}}\\
&{{U}}(0)={{U}}_{0}
\end{aligned}
\right.
\end{equation}
where ${{U}}_{0}=\left(V_{0},V_{1}, P_{0}, P_{1}, u_{0}, \eta_{0}\right)^{T}$ and $\mathcal{A}:\mathcal{D}(\mathcal{A})\subset \mathcal{H}\rightarrow \mathcal{H}$ is given by
\begin{align*}
\mathcal{A}=
\begin{bmatrix}
 0 & \mathrm{I} & 0 & 0 & 0 & 0  \\
  \frac{\alpha}{\rho}\partial_{xx} & 0 & -\frac{\gamma\beta}{\rho}\partial_{xx} & 0 & 0 & 0 \\
  0 & 0 & 0 & \mathrm{I} & 0 & 0 \\
  -\frac{\gamma\beta}{\mu}\partial_{xx} & 0 & \frac{\beta}{\mu}\partial_{xx} & 0 & 0 & 0 \\
  -\frac{\alpha}{m_{1}}\varsigma & 0 & \frac{\gamma\beta}{m_{1}}\varsigma & 0 & -\frac{\xi_{1}}{m_{1}}\mathrm{I} & 0 \\
  \frac{\gamma\beta}{m_{2}}\varsigma & 0 & -\frac{\beta}{m_{2}}\varsigma & 0 & 0 & -\frac{\xi_{2}}{m_{2}}\mathrm{I}
\end{bmatrix}
,
\end{align*}
with $\varsigma \circ \varphi=\varphi_{x}(L)$. The domain of the operator $\mathcal{A}$ is given by
\begin{align*}
\mathcal{D}(\mathcal{A}):=\left\{{{U}}\in \mathcal{H}; \Phi,\Theta\in {H}^{1}_{*}(0,L), V,P \in {H}^{2}(0,L), \Phi(L)=u, \Theta(L)=\eta\right\},
\end{align*}
with ${{U}}=\left(V,\Phi, P, \Theta, u, \eta\right)$.

We now show that operator $\mathcal{A}$ generates a ${C}_{0}-$semigroup $\{\mathcal{S}_{\mathcal{A}}(t)\}_{t\geq0}$ of contractions in the space $\mathcal{H}$. For this purpose, we need the following two lemmas.
\begin{lemma}\label{lem2.1}
The operator $\mathcal{A}$ is dissipative and satisfies that for any ${U}\in \mathcal{D(\mathcal{A})}$,
\begin{align}\label{2.11}
{Re}\langle  {\mathcal{A}}{U},{U}\rangle_{{H}}=-\xi_{1}|u|^{2}-\xi_{2}|\eta|^{2}\leq 0.
\end{align}
\end{lemma}
\begin{proof}
For any ${U}\in \mathcal{D(\mathcal{A})}$, relation \eqref{2.11} can be easily verified by using the inner product in $\mathcal{H}$ and integration by parts.
\end{proof}
\begin{lemma}\label{lem2.2}
The operator $\mathcal{A}$ is bijective and $0\in \varrho(\mathcal{A})$, where $\varrho(\mathcal{A})$ is the resolvent set of $\mathcal{A}$.
\end{lemma}
\begin{proof}
We need to prove that for any ${F}=\left(f_{1},f_{2},f_{3},f_{4},f_{5},f_{6}\right)\in \mathcal{H}$, there exists a ${U}=(V,\Phi,P,\Theta,u,\eta)\in \mathcal{D}(\mathcal{A})$ such that
\begin{align*}
\mathcal{A}{U}={F}.
\end{align*}
Equivalently, we shall consider the existence of unique solution of the system
\begin{equation}\label{2.12}
\left\{
\begin{aligned}
&\Phi=f_{1} \quad\mathrm{in} \quad {H}^{1}_{*}(0,L),\\
& \frac{\alpha}{\rho}V_{xx}-\frac{\gamma\beta}{\rho}=f_{2} \quad \mathrm{in} \quad {L}^{2}(0,L),\\
&\Theta=f_{3} \quad \mathrm{in} \quad {H}^{1}_{*}(0,L),\\
&\frac{\beta}{\mu}P_{xx}-\frac{\gamma\beta}{\mu}=f_{4} \quad \mathrm{in} \quad {L}^{2}(0,L),\\
&-\frac{\alpha}{m_{1}}V_{x}(L,t)+\frac{\gamma\beta}{m_{1}}P_{x}(L,t)-\frac{\xi_{1}}{m_{1}}u=f_{5},\\
&-\frac{\beta}{m_{2}}P_{x}(L,t)+\frac{\gamma\beta}{m_{1}}V_{x}(L,t)-\frac{\xi_{2}}{m_{2}}\eta=f_{6}.
\end{aligned}
\right.
\end{equation}
That is, since
\begin{align}\label{2.13}
\Phi=f_{1},\quad & \quad \Theta=f_{3},
\nonumber\\
u=\Phi(L)=f_{1}(L), \quad &\quad \eta=\Theta(L)=f_{3}(L),
\end{align}
 we need to prove the existence of unique solution of the system
\begin{equation}\label{2.14}
\left\{
\begin{aligned}
& {\alpha}V_{xx}-{\gamma\beta}P_{xx}=\rho f_{2} ,\\
&\beta P_{xx}-\gamma\beta V_{xx}=\mu f_{4} ,\\
&-{\alpha}V_{x}(L,t)+{\gamma\beta}P_{x}(L,t)=m_{1}f_{5}+\xi_{1}f_{1}(L),\\
&-{\beta}P_{x}(L,t)+{\gamma\beta}V_{x}(L,t)=m_{2}f_{6}+{\xi_{2}}f_{3}(L).
\end{aligned}
\right.
\end{equation}

Consider a coercive and continuous and semi-linear operator $\mathcal{G}:\left[{H}^{1}_{*}(0,L)\times{H}^{1}_{*}(0,L)\right]^{2}\rightarrow \mathbb{C}$ defined by
\begin{align*}
\mathcal{G}\left((w_{1},s_{1}),(w_{2},s_{2})\right)=\int^{L}_{0}\left(\alpha w_{1,x}\overline{w}_{2,x}-\gamma\beta w_{1,x}\overline{s}_{2,x}-\gamma\beta s_{1,x}\overline{w}_{2,x}+\beta s_{1,x}\overline{s}_{2,x}\right)dx,
\end{align*}
and a continuous linear functional $\mathcal{F}:{H}^{1}_{*}(0,L)\times{H}^{1}_{*}(0,L)\rightarrow \mathbb{C}$ defined by
\begin{align*}
\mathcal{F}(w,s)=-\int^{L}_{0}\left(\rho f_{2}\overline{w}+\mu f_{4}\overline{s}\right)dx-\left(m_{1}f_{5}+\xi_{1}f_{1}(L)\right)\overline{w}(L)-\left(m_{2}f_{6}+\xi_{2}f_{3}(L)\right)\overline{s}(L).
\end{align*}
By using the Lax-Milgram theorem, we know that there exists a $(V,P)\in {H}^{1}_{*}(0,L)\times {H}^{1}_{*}(0,L)$ satisfying
\begin{align*}
\mathcal{G}\left((V,P),(w,s)\right)=\mathcal{F}(w,s),\quad for\;\;all \quad (w,s)\in {H}^{1}_{*}(0,L)\times{H}^{1}_{*}(0,L).
\end{align*}
From the estimate and \eqref{2.13}, there exists a constant $C>0$ such that
\begin{align*}
\left|\mathcal{G}({U},{F})\right|\leq C\| {U} \|_{\mathcal{H}}\|{F}\|_{\mathcal{H}},
\end{align*}
which implies that
\begin{align*}
\|{U}\|_{\mathcal{H}}\leq \|{F}\|_{\mathcal{H}}\quad \Longleftrightarrow \quad \|{\mathcal{A}}^{-1}{F}\|_{\mathcal{H}}\leqslant \|{F}\|_{\mathcal{H}}.
\end{align*}
Consequence, we conclude that the operator $\mathcal{A}$ generates a ${C}_{0}-$semigroup $\{\mathcal{S}_{\mathcal{A}}(t)\}_{t\geq0}$ of contractions on the space $\mathcal{H}$ by Lumer-Philips theorem \cite{kim1987}. Thus, the proof of the lemma is completed.
\end{proof}

Hence,  using Lemma \ref{lem2.1} and Lemma \ref{lem2.2}, we obtain the well-posedness result.
\begin{theorem}\label{thm2.1}
Let ${U}_{0}\in \mathcal{D(\mathcal{A})}$, there exists a unique solution ${U}(t)=\mathcal{S}_{\mathcal{A}}(t){U}_{0}$ of \eqref{2.10} such that
\begin{align*}
{U}\in{C}\left([0,\infty);\mathcal{D}(\mathcal{A})\right)\cap{C}^{1} \left([0,\infty); \mathcal{H}\right).
\end{align*}
\end{theorem}
\section{Lack of Uniformly Stability}
In this section, we are interested in studying the lack of uniformly stability of the solution of problem \eqref{2.3}-\eqref{2.7}. To show that, we will use the following theorem as a tool.
\begin{theorem}\label{thm3.1}
(\cite{eng2000})
Let $\mathcal{S}(t)={e}^{\mathcal{A}t}$ be a contraction ${C}_{0}-$semigroup on Hilbert space. Then
\begin{align*}
\omega_{0}(\mathcal{S}(t))=\max\{\omega_{ess}(\mathcal{S}(t)),s(\mathcal{A})\},
\end{align*}
where $\omega_{0}(\mathcal{S}(t))$ is the growth bound, $\omega_{ess}(\mathcal{S}(t))$ is the essential growth bound, and $s(\mathcal{A})$ is the spectral bound of the infinitesimal generator $\mathcal{A}$ of $\mathcal{S}(t)$.
\end{theorem}
\begin{theorem}\label{thm3.2}
(\cite[Weyl's Theorem]{weyl1910})
If the difference of the two operator is compact, then the essential spectrum radius are the same.
\end{theorem}
\begin{proposition}\label{them3.3}
(\cite{eng2000})
For a strongly continuous semigroup $\{e^{\mathcal{A}t}\}_{t\leq0}$, the following assertions are equivalent.\\
$(1)$ $\{e^{\mathcal{A}t}\}_{t\leq0}$ is uniformly exponentially stable.\\
$(2)$ $\{e^{\mathcal{A}t}\}_{t\leq0}$ is uniformly stable.
\end{proposition}
\begin{lemma}\label{lem3.4}
Let us fix $\alpha, \beta, \gamma, \rho, \mu$ and the finite interval $E=(a, b)$. Assume that there exists a weak solution to equation
\begin{align}
&\rho V_{tt}-\alpha V_{xx}+\gamma\beta P_{xx}=0,\quad (x,t)\in  (0,L)\times (0,T), \label{3.1}\\
&\mu P_{tt}-\beta P_{xx}+\gamma\beta V_{xx}=0,   \quad (x,t)\in(0,L)\times (0,T),\label{3.2}
\end{align}
If $q(x)=mx+n$, $(m, n\in \mathbb{R})$ and the functions
\begin{align*}
E_{1}(t)&=\int_{E}\left(\rho|V_{t}|^{2}+\alpha_{1}|V_{x}|^{2}+\mu|P_{t}|^{2}+\beta|\gamma V_{x}-P_{x}|^{2}\right)dx,&&t\geq 0,\\
I(x,t)&=\rho|V_{t}(x,t)|^{2}+\alpha_{1}|V_{x}(x,t)|^{2}+\mu|P_{t}(x,t)|^{2}+\beta|(\gamma V_{x}-P_{x})(x,t)|^{2},&& a\geq x\geq b, t\geq 0,
\end{align*}
are integrable in $[a,b]$, then there exists a non-negative constant $M$ satisfying
\begin{align*}
\left|\int^{T}_{0}\left(q(b)I(b,t)-q(a)I(a,t)\right)dt-\int^{T}_{0} m E_{1}(t)dt\right|\leq M\left(E_{1}(T)+E_{1}(0)\right).
\end{align*}
\end{lemma}
\begin{proof}
Multiplying \eqref{3.1} by $q(x)\overline{V}_{x}$, and integrating over $E$, we have
\begin{align}\label{3.5}
\int^{b}_{a} \rho V_{tt} q(x) \overline{V}_{x}dx-\frac{\alpha_{1}}{2}\int^{b}_{a} q(x)\frac{d}{dx}|V_{x}|^{2}dx-\int^{b}_{a} \gamma\beta q(x)(\gamma V_{xx}-P_{xx})\overline{V}_{x}dx=0.
\end{align}
Multiplying \eqref{3.2} by $q(x)\overline{P}_{x}$, and integrating over $E$, we get
\begin{align}\label{3.6}
\int^{b}_{a} \mu P_{tt} q(x) \overline{P}_{x}dx-\int^{b}_{a}\beta q(x)(\gamma V_{xx}-P_{xx})(-\overline{P}_{x})dx=0.
\end{align}
Adding \eqref{2.5} and \eqref{3.6}, we obtain
\begin{align*}
\int^{b}_{a}\left( \rho V_{tt} q(x) \overline{V}_{x}+ \mu P_{tt} q(x) \overline{P}_{x}\right)dx-\frac{\alpha_{1}}{2}\int^{b}_{a} q(x)\frac{d}{dx}|V_{x}|^{2}dx-\frac{\beta}{2}\int^{b}_{a} q(x)\frac{d}{dx}|\gamma V_{x}-P_{x}|^{2}dx=0.
\end{align*}
Integrating by parts over $E$, and using the fact of $q'(x)=m$, we can show
\begin{equation}\label{3.7}
\begin{aligned}
\int^{b}_{a} \left(\rho V_{tt} q(x) \overline{V}_{x}+ \mu P_{tt} q(x) \overline{P}_{x}\right)dx&-\frac{1}{2}\left[\alpha_{1}q(x)|V_{X}|^{2}+\beta q(x)|\gamma V_{x}-P_{x}|^{2}\right]|^{b}_{a}\\
&=\frac{1}{2}\int^{b}_{a}\alpha_{1}m |V_{x}|^{2}+\beta m|\gamma V_{x}-P_{x}|^{2} dx.
\end{aligned}
\end{equation}
Integrating \eqref{3.7} over $[0,T]$, integrating by parts and applying the Fubini theorem, we obtain
\begin{align*}
\left|\int^{T}_{0}\left(q(b)I(b,t)-q(a)I(a,t)\right)dt-\int^{T}_{0} m E_{1}(t)dt\right|=2 \left|Re \int_{E}q(x)\left[\rho V_{t} \overline{V}+\mu P_{t} \overline{P}_{x}\right]|^{T}_{0} dx\right|.
\end{align*}
By using the Young's inequality and the H\"{o}lder's inequality, we can obtain
\begin{align*}
\bigg|\int^{T}_{0}\left(q(b)I(b,t)-q(a)I(a,t)\right)dt&-\int^{T}_{0} m E_{1}(t)dt\bigg|\\
&\leq2\|q\|_{\infty}\left| Re \int_{E}\left[\rho V_{t} \overline{V}+\mu P_{t} \overline{P}_{x}\right]|^{T}_{0} dx\right|\\
&\leq2\|q\|_{\infty}\int^{b}_{a}\left[\rho^{2}|V_{t}|^{2}+|V_{x}|^{2}+\mu^{2}|P_{t}|^{2}+|P_{x}|^{2}\right]|^{T}_{0}dx\\
&\leq2\|q\|_{\infty}\int^{b}_{a}\left[\rho^{2}|V_{t}|^{2}+(1+2\gamma^{2})|V_{x}|^{2}+\mu^{2}|P_{t}|^{2}+2|P_{x}|^{2}\right]|^{T}_{0}dx\\
&\leq M(E_{1}(T)+E_{1}(0)),
\end{align*}
where $M=2\|q\|_{\infty}\max \{\rho, (1+2\gamma^{2})/\alpha_{1}, \mu, 2/\beta\}$. The conclusion follows immediately.
\end{proof}

Then we consider the undamped piezoelectric beams with tip body
\begin{align}
&\rho \widetilde{{V}}_{tt}-\alpha\widetilde{}{ V}_{xx}+\gamma\beta\widetilde{{P}}_{xx}=0,\quad (x,t)\in  (0,L)\times (0,T), \label{3.8}\\
&\mu \widetilde{{P}}_{tt}-\beta \widetilde{{P}}_{xx}+\gamma\beta \widetilde{{V}}_{xx}=0,   \quad (x,t)\in(0,L)\times (0,T),\label{3.9}
\end{align}
with $x\in(0,L), t\leq0$ and boundary conditions
\begin{equation}\label{3.10}
\begin{aligned}
& \widetilde{{V}}(0,t)=\alpha \widetilde{{V}}_{x}(L,t)-\gamma\beta \widetilde{{P}}_{x}(L,t)+m_{1}\widetilde{u}_{t}(t)=0,&t\in(0,+\infty),\\
& \widetilde{{P}}(0,t)=\beta \widetilde{P}_{x}(L,t)-\gamma\beta \widetilde{{V}}_{x}(L,t)+m_{2}\widetilde{\eta}_{t}(t)=0,& t\in(0,+\infty),
\end{aligned}
\end{equation}
with the same initial condition as in \eqref{2.7}
\begin{align}\label{3.11}
\left(\widetilde{V}(x,0), \widetilde{V}_{t}(x,0), \widetilde{P}(x,0), \widetilde{P}_{t}(x,0),\widetilde{u}(0),\widetilde{\eta} (0)\right)=\left(V_{0},V_{1},P_{0},P_{1},u_{0},\eta_{0}\right),\quad x \in (0,L).
\end{align}
Rewrite  system \eqref{3.8}-\eqref{3.11} as the  Cauchy problem
\begin{equation} \label{3.12}
\left\{
\begin{aligned}
&\frac{d}{dt}\widetilde{U}(t)={\mathcal{A}_{0}}{\widetilde{U}}(t)\\
&{\widetilde{{U}}}(0)={{U}}_{0}
\end{aligned}
\right.
\end{equation}
where $\widetilde{U}=(\widetilde{V}, \widetilde{\Phi}, \widetilde{P}, \widetilde{\Theta}, \widetilde{u}, \widetilde{\eta})$, ${{U}}_{0}=\left(V_{0},V_{1}, P_{0}, P_{1}, u_{0}, \eta_{0}\right)^{T}$ and the operator $\mathcal{A}_{0}:\mathcal{D}(\mathcal{A}_{0})\subset \mathcal{H}\rightarrow \mathcal{H}$ is given by
\begin{align*}
\mathcal{A}_{0}=
\begin{bmatrix}
 0 & \mathrm{I} & 0 & 0 & 0 & 0  \\
  \frac{\alpha}{\rho}\partial_{xx} & 0 & -\frac{\gamma\beta}{\rho}\partial_{xx} & 0 & 0 & 0 \\
  0 & 0 & 0 & \mathrm{I} & 0 & 0 \\
  -\frac{\gamma\beta}{\mu}\partial_{xx} & 0 & \frac{\beta}{\mu}\partial_{xx} & 0 & 0 & 0 \\
  -\frac{\alpha}{m_{1}}\varsigma & 0 & \frac{\gamma\beta}{m_{1}}\varsigma & 0 & 0 & 0 \\
  \frac{\gamma\beta}{m_{2}}\varsigma & 0 & -\frac{\beta}{m_{2}}\varsigma & 0 & 0 & 0
\end{bmatrix}
,
\end{align*}
with $\varsigma \circ \varphi=\varphi_{x}(L)$, $\mathcal{D}(\mathcal{A}_{0})=\mathcal{D}(\mathcal{A})$. And the inner product of $\mathcal{H}$ is given by
\begin{align}\label{3.13}
{Re}\langle  {\mathcal{A}_{0}}\widetilde{{U}},\widetilde{{U}}\rangle_{\mathcal{{H}}}= 0,\quad\forall\widetilde{{U}}\in\mathcal{D}(\mathcal{A}_{0}).
\end{align}

Next, we will use the the classical semigroups theory to prove the well-posedness of system \eqref{3.12}.
\begin{theorem}\label{thm3.5}
Let $U_{0}\in \mathcal{D}(\mathcal{A}_{0})$, there exists a unique solution $\widetilde{{U}}(t)=\mathcal{S}_{\mathcal{A}_{0}}(t)\widetilde{{U}}_{0}$ of \eqref{3.12} such that
\begin{align*}
\widetilde{{U}}\in{C}\left([0,\infty);\mathcal{D}(\mathcal{A}_{0})\right)\cap{C}^{1} \left([0,\infty); \mathcal{H}\right).
\end{align*}
\end{theorem}
\begin{proof}
From \eqref{3.12}, the operator ${\mathcal{A}_{0}}$ is skew-hermitian, conservative, closed and densely definite on  $\mathcal{D}(\mathcal{A}_{0})$. Indeed, by using the method similar to Lemma \ref{lem2.2}, we can straightforwardly prove that $0\in\rho(\mathcal{A}_{0})$. Thanks to the Lions theorem, we conclude that the operator $\mathcal{A}_{0}$ is the infinitesimal generator of $C_{0}$-semigroup of contractions  $\{e^{\mathcal{A}_{0}}\}_{t\leq0}$. Then, the conclusion follows immediately.
\end{proof}

\begin{remark}
From \eqref{3.13} and Theorem \ref{thm3.5}, we straightforward see that $\{e^{\mathcal{A}_{0}t}\}_{t\leq0}$ is a $C_{0}$-group unitary by Stone theorem. And the $C_{0}$-group unitary $\{e^{\mathcal{A}_{0}t}\}_{t\leq0}$ has essential spectrum radius 1. The detailed proof content can be found in reference \cite{eng2000,var2020}.
\end{remark}

\begin{lemma}\label{lem3.6}
The set $\{e^{\mathcal{A}t}-e^{\mathcal{A}_{0}t}\}_{t\leq0}$ form a $C_{0}$-semigroup of compact operators.
\end{lemma}
\begin{proof}
For any bounded ${{U}}^{n}_{0}=\left(V^{n}_{0},V^{n}_{1}, P^{n}_{0}, P^{n}_{1}, u^{n}_{0}, \eta^{n}_{0}\right)\in \mathcal{H}$, by using theorem \ref{thm2.1}, we have that
\begin{align*}
U^{n}=e^{\mathcal{A}t}U^{n}_{0}=\left(V^{n},V_{t}^{n}, P^{n}, P_{t}^{n}, u^{n}, \eta^{n}\right)\in \mathcal{H}
\end{align*}
are bounded solutions of \eqref{2.3}-\eqref{2.7}, and thanks to theorem \ref{thm3.5}, we have that
\begin{align*}
\widetilde{U}^{n}=e^{\mathcal{A}_{0}t}\widetilde{U}^{n}_{0}=\left(\widetilde{V}^{n},\widetilde{V_{t}}^{n}, \widetilde{P}^{n}, \widetilde{P_{t}}^{n}, \widetilde{u}^{n}, \widetilde{\eta}^{n}\right)\in \mathcal{H}
\end{align*}
are bounded solutions of \eqref{3.8}-\eqref{3.11}. Indeed, we define
\begin{align*}
\widehat{V}^{n}_{x}=V^{n}_{x}-\widetilde{V}^{n}_{x},\quad \widehat{P}^{n}_{x}=P^{n}_{x}-\widetilde{P}^{n}_{x},&\quad\widehat{V}^{n}_{t}=V^{n}_{t}-\widetilde{V}^{n}_{t},\quad\widehat{P}^{n}_{t}=P^{n}_{t}-\widetilde{P}^{n}_{t},\\
\widehat{u}^{n}=u^{n}-\widetilde{u}^{n},&\quad\widehat{\eta}^{n}={\eta}^{n}-\widetilde{{\eta}}^{n}
\end{align*}
satisfy the system
\begin{align}
&\rho \widehat{V}^{n}_{tt}-\alpha \widehat{V}^{n}_{xx}+\gamma\beta \widehat{P}^{n}_{xx}=0,&&(x,t)\in  (0,L)\times (0,T), \label{3.15}\\
&\mu \widehat{P}^{n}_{tt}-\beta \widehat{P}^{n}_{xx}+\gamma\beta \widehat{V}^{n}_{xx}=0,  && (x,t)\in(0,L)\times (0,T),\label{3.16}
\end{align}
with the boundary conditions
\begin{equation}\label{3.17}
\begin{aligned}
&\widehat{V}^{n}(0,t)=\widehat{P}^{n}(0,t)=0,&& t\in (0,T),\\
&\alpha \widehat{V}^{n}_{x}(L,t)-\gamma\beta \widehat{P}^{n}_{x}(L,t)+\xi_{1} u^{n}(t)+m_{1}\widehat{u}^{n}_{t}(t)=0,&&t\in(0,+\infty), \\
&\beta \widehat{P}^{n}_{x}(L,t)-\gamma\beta \widehat{V}^{n}_{x}(L,t)+\xi_{2} \eta^{n}(t)+m_{2}\widehat{\eta}^{n}_{t}(t)=0,&& t\in(0,+\infty),
\end{aligned}
\end{equation}
and the initial conditions
\begin{align}\label{3.18}
\left(\widehat{V}^{n}(x,0), \widehat{V}^{n}_{t}(x,0), \widehat{P}^{n}(x,0), \widehat{P}^{n}_{t}(x,0), \widehat{u}(0), \widehat{\eta}(0)\right)=\left(0,0,0,0,0,0\right),\quad x \in (0,L).
\end{align}
Then we can show that
\begin{align*}
\left\|\left(e^{\mathcal{A}t}-e^{\mathcal{A}_{0}t}\right)U^{n}_{0}\right\|_{\mathcal{H}}=\int^{L}_{0}\left(\rho|\widehat{V}_{t}|^{2}+\alpha_{1}|\widehat{V}_{x}|^{2}+\mu|\widehat{P}_{t}|^{2}+\beta|\gamma \widehat{V}_{x}-\widehat{P}_{x}|^{2}\right)dx+\frac{m_{1}}{2}|\widehat{u}|^{2}+\frac{m_2}{2}|\widehat{\eta}|^{2},
\end{align*}
and the energy associated with problem \eqref{3.15}-\eqref{3.18} can be defined by
\begin{align*}
\widehat{E}_{n}(t)=\frac{1}{2}\left\|(\widehat{V},\widehat{V}_{t},\widehat{P},\widehat{P}_{t},\widehat{u},\widehat{\eta})\right\|^{2}_{\mathcal{H}}.
\end{align*}

Multiplying \eqref{3.15}, \eqref{3.16} by $\overline{\widehat{V}^{n}_{t}}$, $\overline{\widehat{P}^{n}_{t}}$ respectively, integrating by parts over $(0,L)$, and using the boundary conditions \eqref{3.17}, we can obtain
\begin{align*}
\frac{d}{dt}\bigg[\int^{L}_{0}\left(\rho|\widehat{V}_{t}|^{2}+\alpha_{1}|\widehat{V}_{x}|^{2}+\mu|\widehat{P}_{t}|^{2}+\beta|\gamma \widehat{V}_{x}-\widehat{P}_{x}|^{2}\right)dx+\frac{m_{1}}{2}|\widehat{u}|^{2}+\frac{m_2}{2}|\widehat{\eta}|^{2}\bigg]\\
=-2\xi_{1} Re u^{n}(t)\overline{\widehat{u}^{n}(t)}-2\xi_{2} Re \eta^{n}(t)\overline{\widehat{\eta}^{n}(t)}.
\end{align*}
Integrating the above equation over $[0,T]$ and use the initial conditions \eqref{3.18}, we have
\begin{align*}
\int^{L}_{0}\left(\rho|\widehat{V}_{t}|^{2}+\alpha_{1}|\widehat{V}_{x}|^{2}+\mu|\widehat{P}_{t}|^{2}+\beta|\gamma \widehat{V}_{x}-\widehat{P}_{x}|^{2}\right)dx+\frac{m_{1}}{2}|\widehat{u}|^{2}+\frac{m_2}{2}|\widehat{\eta}|^{2}\\
=-2\xi_{1} Re\int^{T}_{0} u^{n}(t)\overline{\widehat{u}^{n}(t)}dt-2\xi_{2} Re\int^{T}_{0}  \eta^{n}(t)\overline{\widehat{\eta}^{n}(t)}dt.
\end{align*}
Since the operator $\mathcal{A}$ is dissipative and $\mathcal{A}_{0}$ is conservative, we know that energy $\widehat{E}_{n}(t) $ is non-increasing. From \eqref{2.9}, energy $E(T)$ is also non-increasing. That is, $\widehat{E}_{n}(t)\leq \widehat{E}_{n}(0), E(t)\leq E(0), \forall t \geq 0$.

Let $b=L, a=0$ and $ q(x)=x/L$ in Lemma \ref{lem2.2}, we have
\begin{align*}
\int^{L}_{0}\left(\rho|{V}_{t}(L,t)|^{2}+\alpha_{1}|{V}_{x}(L,t)|^{2}+\mu|t{P}_{t}(L,t)|^{2}+\beta|(\gamma {V}_{x}-{P}_{x})(L,t)|^{2}\right)dx
\leq\left(\frac{T}{L}+2M\right)E_{1}(0),
\end{align*}
and
\begin{align*}
\int^{L}_{0}\left(\rho|\widehat{V}_{t}(L,t)|^{2}+\alpha_{1}|\widehat{V}_{x}(L,t)|^{2}+\mu|\widehat{P}_{t}(L,t)|^{2}+\beta|(\gamma \widehat{V}_{x}-\widehat{P}_{x})(L,t)|^{2}\right)dx
\leq\left(\frac{T}{L}+2M\right)E_{1}(0).
\end{align*}
It implies that $V^{n}_{x}(L,t), P^{n}_{x}(L,t)$, $\widehat{V}^{n}_{x}(L,t), \widehat{P}^{n}_{x}(L,t)$, $u^{n}(t), \eta^{n}(t)$, $\widehat{u}^{n}(t), \widehat{\eta}^{n}(t)$ are bounded in $L^{2}(0,T)$. Meanwhile, using the boundary conditions \eqref{2.5}, \eqref{2.6} and \eqref{3.17}, we obtain that $u^{n}_{t}(t), \eta^{n}_{t}(t)$ , $\widehat{u}^{n}_{t}(t), \widehat{\eta}^{n}_{t}(t)$ are also bounded in $L^{2}(0,T)$. We conclude that $u^{n}(t), \eta^{n}(t)$, $\widehat{u}^{n}(t), \widehat{\eta}^{n}(t)$ are bounded in $H^{1}(0,L)$. Thanks to the Rellich-Kondrachov theorem, we obtain that there exist strongly convergent subsequences $u^{nk}(t), \eta^{nk}(t)$, $\widehat{u}^{nk}(t), \widehat{\eta}^{nk}(t)$ in $L^{2}(0,L)$. Then, by using Young's inequality, we have
\begin{align*}
\left\|\left(e^{\mathcal{A}t}-e^{\mathcal{A}_{0}t}\right)U^{n}_{0}\right\|_{\mathcal{H}}\leq &\xi_{1}\left(\int^{T}_{0} \left|u^{nk}(t)\right|^{2}dt+\int^{T}_{0} \left|\widehat{u}^{nk}(t)\right|^{2}dt\right)\\
&+\xi_{2}\left(\int^{T}_{0} \left|\eta^{nk}(t)\right|^{2}dt+\int^{T}_{0} \left|\widehat{\eta}^{nk}(t)\right|^{2}dt\right)\\
&\rightarrow \mathcal{K}, \quad{k\rightarrow\infty},
\end{align*}
where $\mathcal{K}$ is a finite positive constant.
\end{proof}
\begin{theorem}
The $C_0$ semigroup $\{e^{\mathcal{A}t}\}_{t\leq0}$ is not uniformly stable.
\end{theorem}
\begin{proof}
We have know $r_e(e^{\mathcal{A}_{0}t})=1$, and the difference $\{e^{\mathcal{A}t}-e^{\mathcal{A}_{0}t}\}$ is a compact operator. Then, thanks to Theorem \ref{thm3.2}, we obtain
\begin{align*}
r_e(e^{\mathcal{A}t})=r_e(e^{\mathcal{A}_{0}t})=1.
\end{align*}
Using the relationship between $r_e(e^{\mathcal{A}t})$ and $\omega_e(\mathcal{A})$, we obtain
\begin{align*}
\omega_e(\mathcal{A})=0.
\end{align*}
On the other hand, since the operator $\mathcal{A}$ is dissipative, we have $s(\mathcal{A})\leq0$. By using Lemma \ref{thm3.1}, we get
\begin{align*}
\omega_{0}(\mathcal{S}(t))=\max\{\omega_{ess}(\mathcal{S}(t)),s(\mathcal{A})\}=0.
\end{align*}
From the definition of the growth bound in \cite{eng2000}, we have that $\{e^{\mathcal{A}t}\}_{t\leq0}$ is uniformly exponentially stable if and only if $\omega_{0}<0$. So we have $\{e^{\mathcal{A}t}\}_{t\leq0}$ is not uniformly exponentially stable. And it is known that $\{e^{\mathcal{A}t}\}_{t\leq0}$ is a $C_0$-semigroup, which means that it must also be a strongly continuous semigroup. By using Proposition \ref{them3.3}, we conclude that $\{e^{\mathcal{A}t}\}_{t\leq0}$ is not uniformly stable.
\end{proof}
\section{Polynomial stability}
In the previous section, we have shown that the piezoelectric beam system \eqref{2.3}-\eqref{2.8} is  not uniformly stability. In this section, we will state and prove the polynomial stability of our system in this section. It will be achieved by using the following result of Borichev
and Tomilov and two e lemmas.

\begin{theorem}\label{thm4.1}(\cite{bor2010})
Assume that ${\mathcal{S}(t)}_{t\geq0}$ be a bounded $C_{0}-$ semigroup on Hilbert space $H$.  Let $\mathcal{A}$ be the infinitesimal generator of ${\mathcal{S}(t)}_{t\geq0}$ such that $i\mathbb{R}\subset \rho (\mathcal{A})$.  Then, for any $k>0$,  the following conditions are equivalent: \\
$(1)$ $\left\|\left(i\lambda I -\mathcal{A}\right)^{-1}\right\|_{\mathcal{L}(\mathcal{H})}=o\left(|\lambda |^{K}\right),\lambda\rightarrow \infty$;\\
$(2)$ $\left\|\mathcal{S}(t)\mathcal{A}^{-1}\right\|_{\mathcal{L}(\mathcal{H})}=o\left(t^{-\frac{1}{k}}\right),t\rightarrow \infty$.
\end{theorem}

The spectral equation is given by
\begin{align}\label{3.3}
i\lambda U-\mathcal{A}U={F}.
\end{align}
Rewriting \eqref{3.3} in term of its components, we have
\begin{equation}\label{3.4}
\left\{
\begin{aligned}
&i\lambda V-\Phi=F_{1} \quad \mathrm{in} \quad {H}^{1}_{*}(0,L),\\
&i\lambda\rho\Phi-{\alpha}V_{xx}+{\gamma\beta}P_{xx}=\rho F_{2} , \quad \mathrm{in} \quad {L}^{2}(0,L),\\
&i\lambda P-\Theta=F_{3} \quad \mathrm{in} \quad {H}^{1}_{*}(0,L),\\
&i\lambda\mu\Theta-\beta P_{xx}+\gamma\beta V_{xx}=\mu F_{4} , \quad \mathrm{in} \quad {L}^{2}(0,L),\\
&i\lambda m_{1}u+{\alpha}V_{x}(L)-{\gamma\beta}P_{x}(L)+{\xi_{1}}u={m_{1}}F_{5} ,\\
&i\lambda m_{2}\eta+{\beta}P_{x}(L)-{\gamma\beta}V_{x}(L)+{\xi_{2}}\eta={m_{2}}F_{6} ,
\end{aligned}
\right.
\end{equation}
where $\lambda\in \mathbb{R}$ , ${F}=\left(F_{1},F_{2},F_{3},F_{4},F_{5},F_{6}\right)\in \mathcal{H}$ and
\begin{align}\label{4.40-}
\Phi(L)=u,\quad \Theta(L)=\eta.
\end{align}
From \eqref{2.11} and \eqref{4.40-}, we have
\begin{align}\label{4.41-}
\xi_{1}|u|^{2}+\xi_{2}|\eta|^{2}=\xi_{1}|\Phi(L)|^{2}+\xi_{2}|\Theta(L)|^{2}\leq C\|{U}\|_{\mathcal{H}}\|{F}\|_{\mathcal{H}}.
\end{align}

For further proof, we introduce the following functionals and notions.
\begin{align*}
I_{V}&=\rho q(L)|\Phi(L)|^{2}+\alpha_{1}q(L)|V_{x}(L,t)|^{2};\\
I_{P}&=\mu q(L)|\Theta(L)|^{2}+\beta q(L)|(\gamma V_{x}-P_{x})(L,t)|^{2};\\
\mathcal{N}^{2}&=\int^{L}_{0}\rho |\Phi|^{2} dx +\int^{L}_{0}\mu|\Theta|^{2} dx +\int^{L}_{0}\alpha_{1}|V_{x}|^{2}dx +\int^{L}_{0}\beta|\gamma V_{x}-P_{x}|^{2}dx.
\end{align*}
\begin{lemma}\label{lem3.3}
Let us consider $F=(f_{1},f_{2},f_{3},f_{4},f_{5},f_{6})\in\mathcal{H}$, $\lambda\in \mathbb{R}$, and ${U}=(V,\Phi,P,\Theta,u,\eta)\in \mathcal{D}(\mathcal{A})$ such that $(i\lambda{U}-\mathcal{A}{U})={F}$. For $q\in C^{2}([0,L])$, $q(0)=0$, we have
\begin{align*}
&I_{V}+I_{P}-\int^{L}_{0}\rho q_{x}|\Phi|^{2} dx -\int^{L}_{0}\mu q_{x}|\Theta|^{2} dx -\int^{L}_{0}\alpha_{1} q_{x}|V_{x}|^{2}dx -\int^{L}_{0}\beta q_{x}|\gamma V_{x}-P_{x}|^{2}dx\\
=&-R_{1}-R_{2},
\end{align*}
where
\begin{align*}
R_{1}=Re\int^{L}_{0}\left(2\mu q f_{4}\overline{P}_{x}+2\mu q \overline{f}_{3,x}\Theta\right) dx\\
R_{2}=Re\int^{L}_{0}\left(2\rho q f_{2}\overline{V}_{x}-2\rho q \Phi\overline{f}_{1,x} \right)dx.
\end{align*}
\end{lemma}
\begin{proof}
Multiplying \eqref{3.4}$_{2}$ by $q \overline{V}_{x}$ and integrating on $[0,L]$, we get
\begin{align}\label{4.1'}
\int^{L}_{0}\left(- i\lambda \rho \Phi q\overline{V}_{x}+\alpha qV_{xx}\overline{V}_{x}-\gamma\beta qP_{xx}\overline{V}_{x} \right)dx=-\int^{L}_{0}\rho qf_{2}\overline{V}_{x} dx.
\end{align}
Using \eqref{3.4}$_{1}$, we obtain
\begin{align}\label{4.2}
\int^{L}_{0} -i\lambda \rho \Phi q\overline{V}_{x}dx=\int^{L}_{0} \overline{(i\lambda V_{x})}\rho q \Phi dx=\int^{L}_{0}\rho q\Phi\overline{(\Phi_{x}+f_{1,x})}dx.
\end{align}
Multiplying \eqref{3.4}$_{4}$ by $q \overline{P}_{x}$ and integrating on $[0,L]$, we have
\begin{align}\label{4.3}
\int^{L}_{0}\left(- i\lambda \rho\Theta q\overline{P}_{x}+\beta qP_{xx}\overline{P}_{x}-\gamma\beta qV_{xx}\overline{P}_{x} \right)dx=-\int^{L}_{0}\mu qf_{4}\overline{P}_{x} dx.
\end{align}
Using \eqref{3.4}$_{3}$, we obtain
\begin{align}\label{4.4}
\int^{L}_{0}- i\lambda \rho \Theta q\overline{P}_{x}dx=\int^{L}_{0} \overline{(i\lambda P_{x})}\mu q \Theta dx=\int^{L}_{0}\mu q\Theta\overline{(\Theta_{x}+f_{3,x})}dx.
\end{align}
By combining \eqref{4.1'} with \eqref{4.3}, and employing \eqref{4.2} and \eqref{4.4} into it, we conclude that
\begin{align}\label{4.5}
&\int^{L}_{0} \rho q \frac{d}{dx} |\Phi|^{2} dx +\int^{L}_{0} \alpha_{1} q \frac{d}{dx}|V_{x}|^{2} dx +\int^{L}_{0} \mu q \frac{d}{dx}|\Theta|^{2}dx +\int^{L}_{0}\beta q \frac{d}{dx} |\gamma V_{x}-P_{x}|^{2} dx
\nonumber\\
=&Re\int^{L}_{0}\left(-2\mu q f_{4}\overline{P}_{x}-2\mu q \overline{f}_{3,x}\Theta -2\rho q f_{2}\overline{V}_{x}-2\rho q \Phi\overline{f}_{1,x}\right)dx.
\end{align}
Then, integrating by part, we obtain that the relation in Lemma \ref{lem3.3} is correct.
\end{proof}
\begin{lemma}\label{lem4.4}
Let $\mathcal{N}$, $I_{V}, I_{P}$ be functionals defined above, then they satisfy
\begin{align}\label{4.40}
\mathcal{N}^{2}\leq C\left(I_{V}+I_{P}+\|F\|^{2}_{\mathcal{H}}\right),
\end{align}
where $C$ is a constant.
\end{lemma}
\begin{proof}
Let $q(x)=x, x\in [0,L]$. From the result of Lemma \eqref{lem3.3}, we have
\begin{align}\label{4.41}
\mathcal{N}^{2}= L \left(I_{V}+I_{P}\right)-R_{1}-R_{2}.
\end{align}
Since the definition of $R_{1},R_{2}$, we conclude that
\begin{align}\label{4.42}
|R_{1}|\leq C\mathcal{N}\|F\|_{\mathcal{H}}, |R_{2}|\leq C\mathcal{N}\|F\|_{\mathcal{H}}.
\end{align}
Thanks to the estimate \eqref{4.42} and Cauchy-Schwartz inequality, it is straightforward to verify that the relation \eqref{4.40} is valid.
\end{proof}
\begin{theorem}\label{iR}
$i\mathbb{R}\in\rho(\mathcal{A})$, and $\rho(\mathcal{A})$ is the resolvent set of the operator $\mathcal{A}$.
\end{theorem}
\begin{proof}
Since the fact of $0\in \rho(\mathcal{A})$ which we have proved in Section 2, we have that the set
\begin{align*}
\mathcal{M}=\left\{\beta>0:(-i\beta,i\beta)\subset \rho (\mathcal{A})\right\}\neq \emptyset.
\end{align*}
If $\sup\limits_{\beta>0}\mathcal{M}=\infty$, there is nothing to prove. Next, we will consider $\sup {\mathcal{M}}< \infty$ by using reduction to absurdity. Assume that there exists $\lambda>0$ such that $\sup {\mathcal{M}}=\lambda <\infty$. Clearly $\lambda\notin \mathcal{M}$. Therefore, there exist $\lambda_{n} \in \mathcal{M}$ and $\overline{F}_{n}\in \mathcal{H}$ with $\|\overline{F}_{n}\|=1$ such that
\begin{align*}
\left\|(i\lambda_{n} I-\mathcal{A})^{-1}\overline{F}_{n}\right\|_{\mathcal{H}}\rightarrow \infty.
\end{align*}
Let us define $ \overline{U}_{n}=(i\lambda_{n}I-\mathcal{A})^{-1}\overline{F}_{n}$. Then we have that $i\lambda_{n}\overline{U}_{n}-\mathcal{A}\overline{U}_{n}=\overline{F}_{n}$. Denoting $U_{n}=\frac{\overline{U}_{n}}{\|(i\lambda _{n}-\mathcal{A})^{-1}\overline{F}_{n}\|_{\mathcal{H}}}$. Clearly, $U_{n}$ satisfies
\begin{align*}
i\lambda_{n}U_{n}-\mathcal{A}U_{n}=F_{n},
\end{align*}
where $F_{n}==\frac{\overline{F}_{n}}{\|(i\lambda _{n}-\mathcal{A})^{-1}\overline{F}_{n}\|_{\mathcal{H}}}$. Since $\|\overline{F}_{n}\|=1$ and $\|(i\lambda I -\mathcal{A})^{-1}\overline{F}_{n}\|_{\mathcal{H}}\rightarrow \infty$, we have $F_{n}\rightarrow 0$. Taking inner product with $U_{n}$ on $\mathcal{H}$, we obtain
\begin{align*}
i\lambda_{n}\|U_{n}\|^{2}-\langle \mathcal{A}U_{n},U_{n}\rangle _{\mathcal{H}}=\langle F_{n},U_{n}\rangle _{\mathcal{H}}.
\end{align*}
By taking the real part and the fact of $F_{n}\rightarrow 0$, we have that
\begin{align*}
-Re\langle \mathcal{A}U_{n},U_{n}\rangle _\mathcal{H}=Re\langle F_{n},U_{n}\rangle _{\mathcal{H}}\rightarrow 0,
\end{align*}
which implies that
\begin{align*}
\xi_{1}\left|u_{n}(t)\right|^{2}+\xi_{2}\left|\eta_{n}(t)\right|^{2}\rightarrow 0.
\end{align*}
Thanks to \eqref{2.1}, we obtain $\Phi_{n}(L),\Theta_{n}(L)\rightarrow 0$. Using \eqref{2.5} \eqref{2.6} and the fact of $u_{n}(t), \eta_{n}(t)\rightarrow 0$, we have $V_{x,n}(L,t),P_{x,n}(L,t)\rightarrow 0$ which implies that $I_{V}+I_{P}\rightarrow 0$.
By using Lemma \ref{lem4.4} and the fact of $F_{n}, u_{n}(t), \eta_{n}(t)\rightarrow 0$, we conclude that $U_{n}\rightarrow 0$.

This relation contracts with $\|U_{n}\|=1$. Therefore, by using reduction to absurdity, we have proved the theorem.
\end{proof}

Next, by recalling the fact of  Borichev and Tomilov theorem and Lemma \ref{lem4.4}, we prove our result of polynomial stability.
\begin{theorem}\label{thm4.5}
The piezoelectric system \eqref{2.3}-\eqref{2.7} with tip body decays polynomially as
\begin{align*}
\|{U}(t)\|_{\mathcal{H}}\leq \frac{C}{\sqrt{t}}\|{U}_{0}\|_{\mathcal{D}(\mathcal{A})}.
\end{align*}
\end{theorem}
\begin{proof}
Thanks to \eqref{3.4}$_{1}$, \eqref{3.4}$_{3}$ and \eqref{4.41}, we arrive at
\begin{align*}
I_{V}+I_{P}\leq C(1+|\lambda|^{2})\|{U}\|_{\mathcal{H}}\|{F}\|_{\mathcal{H}}+C\|{F}\|^{2}_{\mathcal{H}},
\end{align*}
and by using \eqref{4.40}, we can obtain
\begin{align*}
\mathcal{N}^{2}\leq C(1+|\lambda|^{2})\|{U}\|_{\mathcal{H}}\|{F}\|_{\mathcal{H}}+C\|{F}\|^{2}_{\mathcal{H}}.
\end{align*}
Then, from relations \eqref{4.40} together with the definition of norm in $\mathcal{H}$, we get
\begin{align*}
\|{U}\|_{\mathcal{H}}\leq C|\lambda|^{2}\|{F}\|_{\mathcal{H}}
\end{align*}
for $|\lambda|>1$ large enough.
Finally, we can get the result of polynomial stability  by using Theorem \ref{thm4.1}.
\end{proof}
\subsection*{Acknowledgments}

This work was supported by the National Natural Science Foundation of China [grant number 11771216], the Key Research and Development Program of Jiangsu Province (Social Development) [grant number BE2019725] and the Qing Lan Project of Jiangsu Province.

\section*{Declarations}
\subsection*{Conflict of interest}
 The authors declare that they have no conflict of interest.

\end{document}